\UseRawInputEncoding
\documentclass[12pt,reqno]{amsart}
\usepackage{amsfonts}
\usepackage{amsmath,amssymb,amsthm, mathrsfs,amsopn}
\usepackage{a4wide}
\allowdisplaybreaks
\numberwithin{equation}{section}
 \usepackage[pagewise]{lineno}\linenumbers\nolinenumbers
\usepackage{graphicx}
\usepackage{epstopdf}
\usepackage{float}
\usepackage{stfloats}
\usepackage{subfigure}
\usepackage{soul}
\usepackage{url}
\usepackage{color}
\usepackage[colorlinks, linkcolor=blue, citecolor=blue]{hyperref}

\newtheorem{theorem}{Theorem}[section]

\newtheorem{lemma}[theorem]{Lemma}

\theoremstyle{definition}

\newcommand{\R}{\mathbb{R}}

\newcommand{\eeq}{\end{equation}}
\newcommand{\beq}{\begin{equation*}}

\begin{document}
	
	\title[prescribed the mass solution]{ Normalized Solutions on large smooth domains to the Schr\"{o}dinger equation with potential  and  general
nonlinearity: Mass super-critical case}
	\thanks{The research was supported by National Science Foundation of China *****}
\author{Xiaolu Lin}
	\address{ Xiaolu Lin,~School of Mathematics and Statistics, Central China Normal University, Wuhan 430079, P. R. China}
	\email{ xllin@ccnu.edu.cn}

	\author{Zongyan Lv}
	\address{ Zongyan Lv,~ Center for Mathematical Sciences, Wuhan University of Technology, Wuhan 430070, P. R. China	}
	\email{zongyanlv0535@163.com}
	\date{}

\begin{abstract}
In this paper,  we consider the existence and multiplicity of prescribed mass
solutions to the following nonlinear Schr\"{o}dinger equation with general
nonlinearity: Mass super-critical case:
\begin{equation*}
    \begin{cases}
        -\Delta u+V(x)u+\lambda u=g(u), \\
        \|u\|_2^2=\int|u|^2\mathrm{d}x=c,
    \end{cases}
\end{equation*}
both on large bounded smooth star-shaped domain $\Omega\subset\mathbb{R}^N$ and on $\mathbb{R}^N$, where   $V(x)$ is the potential and the nonlinearity $g(\cdot)$ considered here are very general and of mass super-critical. The standard approach based on the Pohozaev identity to obtain normalized solutions is invalid as the presence of potential $V(x)$.  In addition, our study can be considered as a complement of Bartsch-Qi-Zou (Math Ann 390, 4813--4859, 2024), which has addressed an open problem raised in Bartsch et al. (Commun Partial Differ Equ 46(9):1729¨C1756, 2021).

		\vskip 0.2cm
		\noindent{\bf MR(2010) Subject Classification:} {35J60; 35J20; 35R25}
	\end{abstract}

\maketitle

\section{Introduction}

	The aim of this paper is to  study the existence and multiplicity of solutions for the following nonlinear Schr\"{o}dinger equation:
		\begin{equation}\label{main-eq}
			\begin{cases}
				-\Delta u+V(x)u+\lambda u=g(u)\quad&\text{in}\ \Omega,\\
				u\in H_0^1(\Omega),
			\end{cases}
		\end{equation}
        satisfying the mass constraint
        \begin{equation}\label{constraint}
            \quad\int_{\Omega}|u|^2dx=c,
        \end{equation}
	where $\Omega \subset \mathbb{R}^N$ is either a bounded smooth star-shaped domain or all of $\mathbb{R}^N(N\ge3)$. The mass $c>0$ and $V(x)$ is the potential.  Here the frequency $\lambda$  appears as a Lagrange multiplier.
 It is well known that Eq \eqref{main-eq} comes from the study of standing waves for the nonlinear Schr\"{o}dinger equation(NLS):
		\begin{equation}\label{schro-eq}
				i\frac{\partial\Phi}{\partial t}-V(x)\Phi+\Delta\Phi+g(\Phi)=0, \quad (t,x)\in\mathbb{R}\times\mathbb{R}^N,
		\end{equation}
		The Schr\"{o}dinger equation is a fundamental equation in quantum mechanics, which can be used to describe many physical phenomena, for example, the nonlinear optical problems and the Bose-Einstein condensates, see, e.g. \cite{AEMW95,ESY10}.
		The ansatz $\Phi(x,t)=e^{i\lambda t}u(x)$ for standing waves solutions leads to the equation
		\begin{equation}\label{nls-el-eq}
			-\Delta u+V(x)u+\lambda u=g(u),\quad\text{in}\ \mathbb{R}^{N}.
		\end{equation}
		
			In the equation \eqref{nls-el-eq}, if $\lambda $ is given, we call it fixed frequency problem.
For a fixed frequency $\lambda \in \R  $,	the existence and multiplicity of solutions to \eqref{nls-el-eq} has been investigated in the last two decades by many authors.
The literature in this direction is huge and we do not even make an attempt to summarize it here (see
		e.g. \cite{AIIK19,AlJ22,AlT23,CP09,HLW24} for a survey on almost classical results).

		Recall that the important feature of  \eqref{schro-eq} is that  $L^2$-norm of $\Phi(\cdot,u)$  is  conserved in time, therefore it is natural to consider \eqref{nls-el-eq} with the constraint $\int_{\Omega}|u|^2dx=\alpha$. (see \cite{CL})
		In this case, the mass $\alpha > 0$  is prescribed, while the frequency $\lambda$ is unknown and  comes out as a Lagrange multiplier.  	The existence and properties of these normalized solutions recently has attracted the attention 	of many researchers.

		In the autonomous case, i.e. the potential $V$ is a constant, the NLS equation on the whole space $\mathbb{R}^N$ with combined  power nonlinearity  has attracted a lot of attention since the classical paper Tao-Visan-Zhang \cite{TVZ07} appeared.
		For example, about the situation  $V(x)=0$:
		if the nonlinearities is purely $L^2$-subcritical, then the energy is bounded from below on $S_\alpha$. Thus, for every $\alpha>0$, a ground state can be found as global minimizers of the energy functional constrained in $S_\alpha$, see \cite{Shi17}.
		In the purely $L^2$-supercritical case,  the main  difficulty is that the energy is unbounded from below on $S_\alpha$; however,  Jeanjean \cite{Jean} showed that a normalized groundstate does exist by applying a smart compactness argument, Pohozaev identity and the mountain pass lemma. In \cite{BS17}, Bartsch and Soave further established that the Pohozaev manifold $\{u\in S_\alpha:P(u)=0\}$ is a natural constraint. Moreover, the conditions in \cite{Jean,BS17} can be weaken by reference \cite{Jeanlu20,BM21}.
		Recently,  Soave \cite{Soa20} studied the following equation with mixed  nonlinearities
        \[
			\begin{cases}
				-\Delta u+V(x)u+\lambda u=|u|^{2^*-2}u+\beta|u|^{p-2}u,\\
\int_{\mathbb{R}^N}|u|^2dx=\alpha,
			\end{cases}
		\]
       where $2<p<2+\frac{4}{N}<q<2^*$. Under different assumptions on $\alpha>0,\beta\in\mathbb{R}$, Soave proved several existence and stability/instability by decomposing the Nehari-Pohozaev manifold in a subtle way.
        In particular, for $q=2^*$,
		Soave \cite{Soac20}  further given the existence of ground states, and the existence of mountain-pass solution for $N\ge5$.
		In \cite{JL22}, Jeanjean-Le further proved that, when $N\ge4$, there also
		exist standing waves which are not ground states and are located at a mountain-pass level of the Energy functional. These solutions are unstable by blow-up in finite time.
		For $N=3$, Wei-Wu \cite{WW22} obtianed the existence of solutions of mountain-pass type.
		For more information on the existence of solutions for mixed nonlinear terms of Schr\"{o}dinger equation and systems, we recommend reference \cite{BLZ23,QZ23}.

		We would like to mention here that some paper involved in  dealing with the potential $V(x)$ is  non-constant.   For the case $V\le0$, $V(x)\to0$ as $|x|\to\infty$(decaying potentials), and $f$ is mass subcritical, Ikoma-Miyamoto \cite{IM20} proved that  $e(\alpha)$ is attained for $\alpha>\alpha_0$ and $e(\alpha)$ is not attained for $0<\alpha<\alpha_0$. For the case $V(x)\ge0$, $V(x)\to0$ as $|x|\to\infty$, and allow that the potential has singularities.
        Bartsch-Molle-Rizzi-Verzini \cite{BMRV21} concerned with the existence of normalized solutions based on a new min-max argument. The recent papers \cite{DZ,BHG2,ZZ} and references	therein for new contributions.

		There are only a few approaches and results on the study of normalized solutions in bounded domains \cite{CDCE,BHSG,BHG1,BHG3,DG,DGY,SZ,WJ1}, which dealt with the the autonomous case. In view of the inherent characteristics of the problem with prescribed mass, the so-called Pohozaev manifold is not available when working in bounded domains.  However,  it is worth to point that  the method of these papers above don not work for non-constant $V$.
        Recently, Bartsch-Qi-Zou \cite{BQZ24} considered the existence and properties of normalized solutions with a combination of  mass subcritical and  mass supercritical.

Inspired by \cite{BQZ24}, the aim of  this paper is to investigate the existence of normalized solution
to a class of NLS equations with potential and mass supercritical nonlinearity on bounded
domains.
             By subtle energy estimates, we obtain a normalized solution to \eqref{main-eq} on large bounded smooth star-shaped domains $\Omega_r$ and further obtain the asymptotic behavior (as the radius $r$ tends to infinity), i.e. the existence of normalized solution in whole space.
Next, by studying the asymptotic behavior of the radius $r\to \infty$,  we also obtain  the existence of normalized solution in whole space. Furthermore, based on the existence results in the whole space, we even study the asymptotic behavior of $c\to 0^+$ via a blow up argument.
As far as we know, this manuscript seems to be the first result to normalized solutions of NLS equations with potential and mass supercritical nonlinearity on  $\Omega$ which is large bounded domain even expand to the whole space. Specially, making clear the asymptotic behavior of these normalized solutions for $\alpha\to 0^+$ to the NLS equations  with potential and combined nonlinearities. Last but not least, we consider the bifurcation property.

		Throughout this paper, we will use the following notation: Let $m_+:=\max\{m,0\}$, $m_-:=\min\{m,0\}$ with $m\in\mathbb{R}$. For $\Omega\subset\mathbb{R}^N$, $r>0$ we set
		$
		\Omega_r=\Big\{rx\in\mathbb{R}^N:x\in\Omega\Big\}
		$
		and
		\[
		S_{r,c}:=S_{c}\cap H_0^1(\mathbb{R}^N)=\Big\{u\in H_0^1(\Omega_r):\|u\|_{L^2(\Omega_r)}^2=c\Big\}.
		\]
		without loss of generality,  we assume that $\Omega\subset\mathbb{R}^N$ is a bounded smooth domain with $0\in\Omega$ and star-shaped with respect to 0.  Let $S$ the optimal constant of the Sobolev embedding from $D^{1,2}(\mathbb{R}^N)$ to $L^{2^*}(\mathbb{R}^N)$ see \cite{Aub76}.
		Before stating our main results, we  state our basic assumptions on the potential,
		\begin{description}
			\item[$(V_0)$] $V\in C(\mathbb{R}^N)\cap L^{\frac{N}{2}}(\mathbb{R}^N)$ is bounded and $\|V_-\|_{\frac{N}{2}}<S$.
		\end{description}
		For some results we will require that $V$ is $C^1$ and define
		$\tilde{V}:\mathbb{R}^N\to\mathbb{R}$ by $\tilde{V}(x)=\nabla V(x)\cdot x$.
To  construct a mini-max structure and the compactness
analysis, we  suppose that
\begin{description}
  \item[(G1)] $g: \mathbb{R}\to\mathbb{R}$ is continuous and odd;
  \item[(G2)] There exists some $(\alpha,\beta)\in\mathbb{R}^2_+$ satisfying $2+\frac{4}{N}<\alpha\le\beta<2^*$ such that $0<\alpha G(s)\le g(s)s\le\beta G(s)$ for $s\neq0$, where $G(s)=\int_0^sg(t)\mathrm{d}t$.
  \item[(G3)]  $\lim\limits _{s \rightarrow+\infty} \frac{g(s)}{s^{\beta-1}}=B>0$.
  \item[(G4)]  There exists  some $\mu>0$ and $D>0$ such that
\begin{equation*}
	\lim _{s \rightarrow+\infty} \frac{g^{\prime}(s)}{s^{\beta-2}}=\mu(\beta-1)>0	 \quad \text { and } \quad g^{\prime}(s) \leq D s^{\beta-2}.
\end{equation*}
\end{description}

\begin{theorem}\label{betale0-e>0-Omega}
Suppose that $(V_0)$ and $(G1)-(G3)$ hold, $V(x)\in C^1$ and $\tilde{V}$ is bounded. Then the following hold:
\begin{description}
  \item[(i)] There is $r_c>0$ such that for $c>0$, \eqref{main-eq}-\eqref{constraint} on $\Omega_r$ with $r>r_c$ possess a mountain pass type solution $(\lambda_{r,c},u_{r,c})$ with  the following properties: $u_{r,c}>0$ in $\Omega_r$ and positive energy $E_V(u_{r,c})>0$. Moreover, there exists $C_c>0$ such that
      \[
      \underset{r\to\infty}{\limsup}\,\underset{x\in\Omega_r}{\max}\,u_{r,c}(x)<C_c.
      \]
  \item[(ii)] If in addition $\|\tilde{V}_+\|_{\frac{N}{2}}<2S$ then there exists $\tilde{c}>0$ such that
  \[
  \underset{r\to\infty}{\liminf}\lambda_{r,c}>0\quad\text{for any}\ 0<c<\tilde{c}.
  \]
\end{description}
\end{theorem}

For the passage $r\to\infty$ we require the following condition on $V$.
\begin{description}
  \item[$(V_1)$] $V\in C^1$, $\lim\limits_{|x|\to\infty}V(x)=0$, and there exists $\rho\in\left(0,1\right)$ such that
      \[
      \underset{|x|\to\infty}{\liminf}\underset{y\in B(x,\rho|x|)}{\inf}(x\cdot\nabla V(y))e^{\tau|x|}>0\quad\text{for any }\tau>0.
      \]
\end{description}

\begin{theorem}\label{beta>0-e<0-rn-}
Suppose that  $(V_0)-(V_1)$ and $(G1)-(G3)$ hold,
 $\tilde{V}$ is bounded and $\|\tilde{V}_+\|_{\frac{N}{2}}<2S$, then problem \eqref{main-eq}-\eqref{constraint} with $\Omega=\mathbb{R}^N$ admits for any $0<c<\tilde{c}$, $\tilde{c}$ as in Theorem \ref{betale0-e>0-Omega}(ii), a solution $(\lambda_c,u_c)$ with $u_c>0$ , $\lambda_c>0$, and $E_V(u_c)>0$, $\lim\limits_{c\to\infty}E_V(u_c)=\infty$.
\end{theorem}

The rest of this manuscript is organized as follows.
In Section 2, we  collect some notations and some preliminary results   which will be used in this paper.
In Section 3, we obtain the existence and properties of the existence of mountain-pass type solutions.
The compactness of the normalized solution in $B_r$ as $r\to\infty$  and the proof of Theorem \ref{beta>0-e<0-rn-} are given in the Section 4.

\section{Preliminaries.}
\setcounter{equation}{0}
\setcounter{theorem}{0}	

This section is devoted to collect some preliminary results which will be used in this paper.
Let us first introduce the Gagliardo-Nirenberg inequality, see \cite{Wei82}.
\begin{lemma}
For any $s\in\left(2,2^*\right)$, there is a constant $C_{N,s}$ depending on
 s such that
\[
\|u\|_s^s\le C_{N,s}\|u\|_2^{2s-N(s-2)}\|\nabla u\|_2^{\frac{N(s-2)}{2}},\quad\forall\ u\in H^1(\mathbb{R}^N),
\]
where $C_{N,s}$ be the best constant.
\end{lemma}

Next we recall the Monotonicity trick  \cite{BCJS10,CJS22}.
\begin{theorem}\label{mp-th}(Monotonicity trick)
Let $(E,\langle\cdot,\cdot\rangle)$ and $(H,(\cdot,\cdot))$ be two infinite-dimensional Hilbert spaces and assume there are continuous injections
\[
E\hookrightarrow H\hookrightarrow E'
\]
Let
\[
\|u\|^2=\langle u,u\rangle,\quad |u|^2=(u,u)\quad\text{for}\,u\in E,
\]
and
\[
S_\mu=\{u\in E:|u|^2=\mu\},\quad T_u S_\mu=\{v\in E:(u,v)=0\}\quad\text{for}\,\mu\in\left(0,+\infty\right).
\]
Let $I\subset\left(0,+\infty\right)$ be an interval and consider a family of $C^2$ functionals $\Phi_\rho: E\to\mathbb{R}$ of the form
\[
\Phi_\rho(u)=A(u)-\rho B(u)\quad\text{for}\,\rho\in I,
\]
with $B(u)\ge0$ for every $u\in E$, and
\[
A(u)\to+\infty\quad\text{or}\quad B(u)\to+\infty\quad\text{as}\,u\in E\,\text{and}\,\|u\|\to+\infty.
\]
Suppose moreover that $\Phi'_\rho$ and $\Phi''_\rho$ are H\"{o}lder continuous, $\tau\in\left(0,1\right]$, on bounded sets in the following sense: for every $R>0$ there exists $M=M(R)>0$ such that
\begin{equation}
\|\Phi'_\rho(u)-\Phi'_\rho(v)\|\le M\|u-v\|^\tau\quad \|\Phi''_\rho(u)-\Phi''_\rho(v)\|\le M\|u-v\|^\tau
\end{equation}
for every $u,v\in B(0,R)$. Finally, suppose that there exist $w_1,w_2\in S_\mu$ independent of $\rho$ such that
\[
c_\rho:=\underset{\gamma\in\Gamma}{\inf}\underset{t\in\left[0,1\right]}{\max}\Phi_\rho(\gamma(t))
>\max\{\Phi_\rho(w_1),\Phi_\rho(w_2)\}\quad\text{for all}\ \rho\in I,
\]
where
\[
\Gamma=\{\gamma\in C(\left[0,1\right],S_\mu):\gamma(0)=w_1,\gamma(1)=w_2\}.
\]
Then for almost every $\rho\in I$ there exists a sequence $\{u_n\}\subset S_\mu$ such that
\begin{description}
  \item[(i)] $\Phi_\rho(u_n)\to c_\rho$,
  \item[(ii)] $\Phi'_\rho|_{S_\mu}(u_n)\to0$,
  \item[(iii)] $\{u_n\}$ is bounded in $E$.
\end{description}
\end{theorem}

From the assumptions (G1) and (G2), we deduce that for all $t\in\mathbb{R}$ and $s\ge0$,
\[
\begin{cases}
s^\beta G(t)\le G(ts)\le s^\alpha G(t),&\text{if}\ s\le1,\\
s^\alpha G(t)\le G(ts)\le s^\beta G(t),&\text{if}\ s\ge1.
\end{cases}
\]
Moreover, there exists some constants $C_1,C_2$ and $d,d'\in\{\alpha,\beta\}$ such that for all $s\in\mathbb{R}$,
\begin{equation}\label{G-compare}
C_1|s|^{d'}=:C_1\min\{|s|^\alpha,|s|^\beta\}\le G(s)\le C_2\max\{|s|^\alpha,|s|^\beta\}:=C_2|s|^d,
\end{equation}
and
\begin{equation}\label{g-G}
(\frac{\alpha}{2}-1)G(s)\le\frac{1}{2}g(s)s-G(s)
\le(\frac{\beta}{2}-1)G(s)
\le(\frac{\beta}{2}-1)C_2(|s|^\alpha+|s|^\beta).
\end{equation}
 Consider
\begin{equation}\label{main-eq-omega}
\begin{cases}
-\Delta u+V(x)u+\lambda u=g(u)&\text{in}\ \Omega_r,\\
u\in H^{1}_{0}(\Omega_r),\ \int_{\Omega_r}|u|^2dx=c.
\end{cases}
\end{equation}
where  the mass $c>0$  and the frequency $\lambda$ is unknown.
The energy functional $E_r:H_0^1(\Omega_r)\to\mathbb{R}$ is defined by
\[
E_r(u)=\frac{1}{2}\int_{\Omega_r}|\nabla u|^2dx
+\frac{1}{2}\int_{\Omega_r}V(x)u^2dx
-\int_{\Omega_r}G(u)dx,
\]
and the mass constraint manifold is defined by
\[
S_{r,c}=\Big\{u\in H_0^1(\Omega_r):\|u\|_2^2=c\Big\}.
\]

\section{Proof of Theorem \ref{betale0-e>0-Omega}}
\setcounter{equation}{0}
\setcounter{theorem}{0}	
In this section, we always assume that the assumptions of Theorem \ref{betale0-e>0-Omega} hold. In
order to obtain a bounded Palais-Smale sequence, we will use the monotonicity trick.
For $\frac{1}{2}\le s\le1$, the energy functional $E_{r,s}:S_{r,c}\to\mathbb{R}$ is defined by
\[
E_{r,s}(u)=\frac{1}{2}\int_{\Omega_r}|\nabla u|^2dx +\frac{1}{2}\int_{\Omega_r}V(x)u^2dx
-s\int_{\Omega_r}G(u)dx.
\]
Note that  if $u\in S_{r,c}$ is a critical point of $E_{r,c}$, then there exists $\lambda\in\mathbb{R}$ such that $(\lambda,u)$ is a solution of the equation
\begin{equation}\label{main-eq-s-omega}
\begin{cases}
-\Delta u+V(x)u+\lambda u=sg(u)&\text{in}\ \Omega_r,\\
u\in H^{1}_{0}(\Omega_r),\ \int_{\Omega_r}|u|^2dx=c.
\end{cases}
\end{equation}

\begin{lemma}\label{betale0-mp-g}
For any $c>0$, there exists $r_c>0$ and $u^0,u^1\in S_{r,c}$ such that
\begin{description}
  \item[(i)] $E_{r,s}(u^1)\le0$ for any $r>r_c$ and $s\in\left[\frac{1}{2},1\right]$,
  \[
  \|\nabla u^0\|_2^2<\bigg(\frac{2d}{NC_{N,d}C_2(d-2)}
(1-\|V_-\|_{\frac{N}{2}}S^{-1})c^{\frac{d(N-2)-2N}{4}}\bigg)^{\frac{4}{N(d-2)-4}}<\|\nabla u^1\|_2^2
  \]
  and
  \[
  E_{r,s}(u^0)<\frac{(N(d-2)-4)(1-\|V_-\|_{\frac{N}{2}}S^{-1})}{2N(d-2)}
\bigg(\frac{2d(1-\|V_-\|_{\frac{N}{2}}S^{-1})c^{\frac{2d-N(d-2)}{4}}}{NC_{N,d}C_2(d-2)}
\bigg)^{\frac{4}{N(d-2)-4}}.
  \]
  \item[(ii)] If $u\in S_{r,c}$ satisfies
  \[
  \|\nabla u\|_2^2=\bigg(\frac{2d}{NC_{N,d}C_2(d-2)}
(1-\|V_-\|_{\frac{N}{2}}S^{-1})c^{\frac{d(N-2)-2N}{4}}\bigg)^{\frac{4}{N(d-2)-4}},
  \]
  then there holds
  \[
  E_{r,s}(u)\ge\frac{(N(d-2)-4)(1-\|V_-\|_{\frac{N}{2}}S^{-1})}{2N(d-2)}
\bigg(\frac{2d(1-\|V_-\|_{\frac{N}{2}}S^{-1})c^{\frac{2d-N(d-2)}{4}}}{NC_{N,d}C_2(d-2)}
\bigg)^{\frac{4}{N(d-2)-4}}.
  \]
  \item[(iii)] Set
  \[
  m_{r,s}(c)=\underset{\gamma\in\Gamma_{r,c}}{\inf}\,
  \underset{t\in\left[0,1\right]}{\sup}\,E_{r,s}(\gamma(t)),
  \]
  with
  \[
  \Gamma_{r,c}=\Big\{\gamma\in C(\left[0,1\right],S_{r,c}):\gamma(0)=u^0,\gamma(1)=u^1\Big\}.
  \]
  Then
  \[
  \frac{(N(d-2)-4)(1-\|V_-\|_{\frac{N}{2}}S^{-1})}{2N(d-2)}
\bigg(\frac{2d(1-\|V_-\|_{\frac{N}{2}}S^{-1})c^{\frac{2d-N(d-2)}{4}}}{NC_{N,d}C_2(d-2)}
\bigg)^{\frac{4}{N(d-2)-4}}\le m_{r,s}(c)<H_c.
  \]
  Here $H_c:=\max_{t\in\mathbb{R}^+}h(t)$, where
  \[
  h(t)=\frac{1}{2}\bigg(1+\|V\|_{\frac{N}{2}}S^{-1}\bigg)t^2\theta c
-C_1t^{\frac{3(d'-2)}{2}}c^{\frac{d'}{2}}|\Omega|^{\frac{2-d'}{2}}.
  \]
\end{description}
\end{lemma}
\begin{proof}
Clearly the set $S_{r,c}$ is path connected.  Let  $v_t(x)=t^{\frac{3}{2}}v_1(tx)$, where $v_1\in S_{r,c}$ is the positive normalized eigenfunction of $-\Delta$ with Dirichlet boundary condition in $\Omega$ associated to $\theta$.
We can easily verify  $v_t(x)\in S_{1,c}$. Owing to
\[
\int_{\Omega}|\nabla v_1|^2\mathrm{d}x=\theta c \quad\text{and}\quad\int_{\Omega}|v_1|^{d'}\mathrm{d}x\ge c^{\frac{d'}{2}}|\Omega|^{\frac{2-d'}{2}}
\]
For $x\in\Omega_{\frac{1}{t}}$ and $t>0$, from \eqref{G-compare} there holds
\begin{eqnarray}\label{e-vt-s-h}
E_{\frac{1}{t},s}(v_t)&\le&\frac{1}{2}t^2\bigg(1+\|V\|_{\frac{3}{2}}S^{-1}\bigg)\int_{\Omega}|\nabla v_1|^2\mathrm{d}x
-C_1t^{\frac{3(d'-2)}{2}}\int_{\Omega}|v_1|^{d'}\mathrm{d}x\\
&\le& h(t)\notag.
\end{eqnarray}
Note that since $2+\frac{4}{N}<d'<2^*$, there exist $0<T_c<t_c$ such that $h(t_c)=0$, $h(t)<0$ for any $t>t_c$, $h(t)>0$ for any $0<t<t_c$ and $h(T_c)={\max}_{t\in\mathbb{R}^+} h(t)$. As a consequence, there holds
\begin{equation}\label{htalpha0}
E_{r,s}(v_{t_c})=E_{\frac{1}{t_{c}},s}(v_{t_c})\le h(t_c)=0
\end{equation}
for any $r\ge\frac{1}{t_c}$ and $s\in\left[\frac{1}{2},1\right]$. Moreover, there exists $0<t_1<T_c$ such that for any $t\in\left[0,t_1\right)$,
\begin{equation}\label{htupbdd}
h(t)<h(t_1)\le\frac{(N(d-2)-4)(1-\|V_-\|_{\frac{N}{2}}S^{-1})}{2N(d-2)}
\bigg(\frac{2d(1-\|V_-\|_{\frac{N}{2}}S^{-1})c^{\frac{2d-N(d-2)}{4}}}{NC_{N,d}C_2(d-2)}
\bigg)^{\frac{4}{N(d-2)-4}}.
\end{equation}

On the other hand, it follows from \eqref{G-compare}, the Gagliardo-Nirenberg inequality and the H\"{o}lder inequality that
\begin{equation}\label{ers-bel}
E_{r,s}(u)\ge\frac{1}{2}\bigg(1-\|V_-\|_{\frac{3}{2}}S^{-1}\bigg)\int_{\Omega_r}|\nabla u|^2dx-\frac{C_dC_2c^{\frac{6-d}{4}}}{d}\bigg(\int_{\Omega_r}|\nabla u|^2dx\bigg)^{\frac{3(d-2)}{4}}.
\end{equation}
Let
\[
f(t):=\frac{1}{2}\bigg(1-\|V_-\|_{\frac{3}{2}}S^{-1}\bigg)t
-\frac{C_dC_2c^{\frac{6-d}{4}}}{d}t^{\frac{3(d-2)}{4}}
\]
and
\[
\tilde{t}:=\bigg(\frac{2d}{NC_{N,d}C_2(d-2)}
(1-\|V_-\|_{\frac{N}{2}}S^{-1})c^{\frac{d(N-2)-2N}{4}}\bigg)^{\frac{4}{N(d-2)-4}}.
\]
Then $f$ is increasing in $(0,\tilde{t})$ and decreasing in $(\tilde{t},\infty)$, and
\[
f(\tilde{t})=\frac{(N(d-2)-4)(1-\|V_-\|_{\frac{N}{2}}S^{-1})}{2N(d-2)}
\bigg(\frac{2d(1-\|V_-\|_{\frac{N}{2}}S^{-1})c^{\frac{2d-N(d-2)}{4}}}{NC_{N,d}C_2(d-2)}
\bigg)^{\frac{4}{N(d-2)-4}}.
\]
For $r>\tilde{r}_c:=
\max\Big\{\frac{1}{t_1},\sqrt{\frac{2\theta c}{\tilde{t}}}\Big\}$, we have $v_{\frac{1}{\tilde{r}_c}}\in S_{\tilde{r}_c,c}\subset S_{r,c}$ and
\begin{equation}\label{nabla-u0}
\|\nabla v_{\frac{1}{\tilde{r}_c}}\|_2^2=\bigg(\frac{1}{\tilde{r}_c}\bigg)^2\|\nabla v_1\|_2^2<\bigg(\frac{2d}{NC_{N,d}C_2(d-2)}
(1-\|V_-\|_{\frac{N}{2}}S^{-1})c^{\frac{d(N-2)-2N}{4}}\bigg)^{\frac{4}{N(d-2)-4}}.
\end{equation}
Moreover, there holds
\begin{equation}\label{e-ralpha-s-ht1}
E_{\tilde{r}_c,s}(v_{\frac{1}{\tilde{r}_c}})\le h\bigg(\frac{1}{\tilde{r}_c}\bigg)<h(t_1).
\end{equation}
Setting $u^0=v_{\frac{1}{\tilde{r}_c}}$, $u^1=v_{t_c}$ and
\[
r_c=\max\Big\{\frac{1}{t_c},\tilde{r}_c\Big\},
\]
the statement (i) holds due to  \eqref{htalpha0}-\eqref{htupbdd} and \eqref{nabla-u0}-\eqref{e-ralpha-s-ht1}.
(ii) Note that statement (ii) holds by \eqref{ers-bel} and a direct calculation.
(iii) Since $E_{r,s}(u^1)\le0$ for any $\gamma\in\Gamma_{r,c}$, we have
\[
\|\nabla \gamma(0)\|_2^2<\tilde{t}<\|\nabla\gamma(1)\|_2^2.
\]
It then follows from \eqref{ers-bel} that
\[
\underset{t\in\left[0,1\right]}{\max}E_{r,s}(\gamma(t))\ge f(\tilde{t})=\frac{(N(d-2)-4)(1-\|V_-\|_{\frac{N}{2}}S^{-1})}{2N(d-2)}
\bigg(\frac{2d(1-\|V_-\|_{\frac{N}{2}}S^{-1})c^{\frac{2d-N(d-2)}{4}}}{NC_{N,d}C_2(d-2)}
\bigg)^{\frac{4}{N(d-2)-4}}
\]
for any $\gamma\in\Gamma_{r,c}$, hence the first inequality  holds. For the second inequality,  we define a path $\gamma_0\in\Gamma_{r,c}$ by
\[
\gamma_0(\tau):\Omega_r\to\mathbb{R},\quad x\mapsto\bigg(\tau t_c+(1-\tau)\frac{1}{\tilde{r}_c}\bigg)^{\frac{N}{2}}v_1\bigg(\bigg(\tau t_c+(1-\tau)\frac{1}{\tilde{r}_c}\bigg)x\bigg)
\]
for $\tau\in\left[0,1\right]$.
\end{proof}

In view of Lemma \ref{betale0-mp-g}, the energy functional $E_{r,s}$ possesses the mountain pass geometry.  For the sequence obtained from  Theorem \ref{mp-th}, we will show its convergence in the next theorem.
\begin{theorem}\label{beta>0-s-omega}
For $r>r_c$, where $r_c$ is defined in Lemma \ref{betale0-mp-g}. Problem \eqref{main-eq-s-omega} admits a solution $(\lambda_{r,s},u_{r,s})$ for almost every $s\in\left[\frac{1}{2},1\right]$. Moreover, $u_{r,s}\ge0$ and $E_{r,s}(u_{r,s})=m_{r,s}(c)$.
\end{theorem}
\begin{proof}
For fixed $c>0$, let us apply Theorem \ref{mp-th} to $E_{r,s}$ with $\Gamma_{r,c}$ given in Lemma \ref{betale0-mp-g}-(iii),
\[
A(u)=\frac{1}{2}\int_{\Omega}|\nabla u|^2\mathrm{d}x
+\frac{1}{2}\int_{\Omega}V(x)u^2\mathrm{d}x
\]
and
\[
B(u)=\int_{\Omega}G(u)\mathrm{d}x.
\]
Thanks to Lemma \ref{betale0-mp-g}, the assumptions in Theorem \ref{mp-th} hold. Hence, for almost every $s\in\left[\frac{1}{2},1\right]$, there exists a nonnegative bounded Palais-Smale sequence $\{u_n\}$:
\[
J_{r,s}(u_n)\to m_{r,s}(c)\quad\text{and}\quad J'_{r,s}(u_n)|_{T_{u_n}S_{r,c}}\to0,
\]
where $T_{u_n}S_{r,c}$ denoted the tangent space of $S_{r,c}$ at $u_n$. Let us assume that $\int_{\Omega_r}|\nabla u_n|^2\mathrm{d}x\to\Lambda$, note that
\begin{equation}\label{un-ers-ge}
J'_{r,s}(u_n)+\lambda_nu_n\to0 \quad\text{in}\ H^{-1}(\Omega_r)
\end{equation}
and
\[
\lambda_n=-\frac{1}{c}\bigg(\int_{\Omega_r}|\nabla u_n|^2\mathrm{d}x+\int_{\Omega_r}V(x) u^2_n\mathrm{d}x-s\int_{\Omega_r}G(u_n)\mathrm{d}x\bigg).
\]
is bounded.
Furthermore, there exist $u_0\in H_0^1(\Omega_r)$ and $\lambda\in\mathbb{R}$ such that up to a subsequence,
\[
\lambda_n\to\lambda,\quad u_n\rightharpoonup u_0\ \text{in}\ H_0^1(\Omega_r)\quad \text{and}\ u_n\to u_0\ \text{in}\ L^t(\Omega_r)\ \text{for all }2\le t<2^*,
\]
and $u_0$ satisfies
\begin{equation}\label{eq-u0-ge}
\begin{cases}
-\Delta u_0+Vu_0+\lambda u_0=sg(u_0)&\text{in}\,\Omega_r,\\
u_0\in H_0^1(\Omega_r),\quad\int_{\Omega_r}|u_0|^2\mathrm{d}x=c.
\end{cases}
\end{equation}
In view of \eqref{un-ers-ge}, we have
\[
J'_{r,s}(u_n)u_0+\lambda_n\int_{\Omega_r}u_nu_0\mathrm{d}x\to0\quad\text{and}\quad J'_{r,s}(u_n)u_n+\lambda_nc\to0\quad\text{as}\quad n\to\infty,
\]
and
\[
J'_{r,s}(u_0)u_n+\lambda\int_{\Omega_r}u_nu_0\mathrm{d}x\to0,\quad J'_{r,s}(u_0)u_0+\lambda c\to0\quad\text{as}\quad n\to\infty.
\]
Owing to
\[
\underset{n\to\infty}{\lim}\int_{\Omega_r}V(x)u_n^2\mathrm{d}x=\int_{\Omega_r}V(x)u_0^2\mathrm{d}x.
\]
We conclude that $u_n\to u_0$ in $H_0^1(\Omega_r)$ as $n\to\infty$. Consequently, $J_{r,s}(u_0)=m_{r,s}(c)$ and $u_0$ is a nonnegative normalized solution to \eqref{main-eq-s-omega}.
\end{proof}

Next, we will further establish a uniform estimate  in order to get a solution of problem \eqref{main-eq-omega}.
\begin{lemma}\label{nabla-u-bdd}
Suppose that $(\lambda,u)\in\mathbb{R}\times S_{r,c}$ is a solution of problem \eqref{main-eq-s-omega} established in Theorem \ref{beta>0-s-omega} for some $r$ and $s$, then
\[
\int_{\Omega_r}|\nabla u|^2\mathrm{d}x\le \frac{4N}{(N(\alpha-2)-4)}\bigg(\frac{\alpha-2}{2}H_c
+c\bigg(\frac{1}{2N}\|\nabla V\cdot x\|_\infty+\frac{\alpha-2}{4}\|V\|_\infty\bigg)\bigg).
\]
\end{lemma}
\begin{proof}
Based on the fact that $(\lambda,u)\in\mathbb{R}\times S_{r,\alpha}$ is a solution of problem \eqref{main-eq-s-omega}, we know
\begin{equation}\label{eq-sol-bdd-ge}
\int_{\Omega_r}|\nabla u|^2\mathrm{d}x+\int_{\Omega_r}V(x)u^2\mathrm{d}x
+\lambda\int_{\Omega_r}|u|^2\mathrm{d}x
=s\int_{\Omega_r}g(u)u\mathrm{d}x
.
\end{equation}
In addition, the pohozaev identity leads to
\begin{eqnarray}\label{pohozaev-ge}
&&\frac{N-2}{2N}\int_{\Omega_r}|\nabla u|^2\mathrm{d}x+\frac{1}{2N}\int_{\partial\Omega_r}|\nabla u|^2(x\cdot\textbf{n})\mathrm{d}\sigma+\frac{1}{2N}\int_{\Omega_r}\tilde{V}u^2\mathrm{d}x
\notag\\
&&\quad\quad=-\frac{1}{2}\int_{\Omega_r}Vu^2\mathrm{d}x-\frac{\lambda}{2}\int_{\Omega_r}|u|^2\mathrm{d}x
+s\int_{\Omega_r}G(u)\mathrm{d}x,
\end{eqnarray}
where $\textbf{n}$ denotes the outward unit normal vector on $\partial\Omega_r$.
Then via the inequality \eqref{g-G}, \eqref{eq-sol-bdd-ge} and \eqref{pohozaev-ge}, we deduce that
\begin{eqnarray*}
&&\frac{1}{N}\int_{\Omega_r}|\nabla u|^2\mathrm{d}x-\frac{1}{2N}\int_{\partial\Omega_r}|\nabla u|^2(x\cdot\textbf{n})\mathrm{d}\sigma-\frac{1}{2N}\int_{\Omega_r}(\nabla V\cdot x)u^2\mathrm{d}x\\
&=&\frac{1}{2}\int_{\mathbb{R}^N}sg(u)u\mathrm{d}x-\int_{\mathbb{R}^N}sG(u)\mathrm{d}x\\
&\ge&\frac{\alpha-2}{2}\int_{\mathbb{R}^N}sG(u)\mathrm{d}x\\
&=&\frac{\alpha-2}{2}\bigg(\frac{1}{2}\int_{\Omega_r}|\nabla u|^2\mathrm{d}x
+\frac{1}{2}\int_{\Omega_r}V|u|^2\mathrm{d}x-m_{r,s}(c)\bigg).
\end{eqnarray*}
Recall that  $\Omega_r$ is starshaped with respect to 0, so $x\cdot\textbf{n}\ge0$ for any $x\in\partial\Omega_r$, thereby,
\begin{eqnarray*}
\frac{\alpha-2}{2}m_{r,s}(c)&\ge&\frac{\alpha-2}{2}\bigg(\frac{1}{2}\int_{\Omega_r}|\nabla u|^2\mathrm{d}x
+\frac{1}{2}\int_{\Omega_r}V|u|^2\mathrm{d}x\bigg)\\
&&\quad-\bigg(\frac{1}{N}\int_{\Omega_r}|\nabla u|^2\mathrm{d}x
-\frac{1}{2N}\int_{\partial\Omega_r}|\nabla u|^2(x\cdot\textbf{n})\mathrm{d}\sigma\\
&&\quad-\frac{1}{2N}\int_{\Omega_r}(\nabla V\cdot x)u^2\mathrm{d}x\bigg)\\
&\ge&\frac{N(\alpha-2)-4}{4N}\int_{\Omega_r}|\nabla u|^2\mathrm{d}x-c\bigg(\frac{1}{2N}\|\nabla V\cdot x\|_\infty+\frac{\alpha-2}{4}\|V\|_\infty\bigg).
\end{eqnarray*}
According to  $m_{r,s}(c)<H_c$, the proof of Lemma \ref{nabla-u-bdd} is now complete.
\end{proof}

The following is the priori bound to the solutions of problem \eqref{main-eq-omega}.
\begin{lemma}\label{ur-ubdd}
Let $\{(\lambda_r,u_r)\}$ be  nonnegative solutions of  \eqref{main-eq-omega}  with $\|u_r\|_{H^1}\le C$, where $C>0$ is independent of $r$, then $\limsup\limits_{r\to\infty}\,\|u_r\|_\infty<\infty$.
\end{lemma}

\begin{proof}
    We assume by the contradiction  that there are $\{u_r\}$ and $x_r\in\Omega_r$ such that
    \[
    M_r:=\underset{x\in\Omega_r}{\max}\,u_r(x)=u_r(x_r)\to\infty\quad\text{as}\ r\to\infty.
    \]
    Now we perform a rescaling, defining
    \[
\omega_r:=\frac{u_r(x_r+\tau_r x)}{M_r}\quad\text{for}\ x\in\Sigma^r:=\{x\in\mathbb{R}^N:x_r+\tau_r x\in\Omega_r\},
    \]
    where $\tau_r=M_r^{-\frac{\beta-2}{2}}$. From the uniform boundedness of $\|\nabla u_r\|_2$ and the boundedness is independent of $r$.  Through a delicate calculation, we get that $\tau_r\to0$, $\|\omega_r\|_{L^\infty(\Sigma^r)}\le1$ and $\omega_r$ satisfies
    \begin{align}\label{equ:241204-e1}
        -\Delta\omega_r+\tau_r^2V(x_r+\tau_rx)\omega_r+\tau^2_r\lambda_r\omega_r
        =\frac{g(M_r\omega_r)}{M_r^{\beta-1}}\quad\text{in}\ \Sigma^r.
    \end{align}
It follows from \eqref{main-eq-omega},  the Gagliardo-Nirenberg inequality,  the Sobolev inequatlity and $\left\|u_r\right\|_{H^1} \leq C$ that  the sequence $\left\{\lambda_r\right\}$ is bounded. According to   $L^p$ estimates  (see [20, Theorem 9.11]), we can deduce that $\omega_r \in W_{\text {loc }}^{2, p}(\Sigma)$ and $\|\omega_r\|_{W_{\text {loc }}^{2, p}(\Sigma)} \leq C$ for any $p>1$ and $\Sigma:=\lim\limits _{r \rightarrow \infty} \Sigma^r$. Using Sobolev embedding theorem, we can obtain that $\omega_r \in C_{\mathrm{loc}}^{1, t}(\Sigma)$ for some $t \in(0,1)$ and $\|\omega_r\|_{C_{\text {loc }}^{1, t}(\Sigma)} \leq C$.   Therefore,  applying  the Arzela-Ascoli theorem, we find that there exists $v$ such that up to a subsequence
 \[\omega_r\to\omega\quad\text{in}\  C_{loc}^\beta(\Sigma).\]
 where $\Sigma=\lim_{r\to\infty}\Sigma^r$ is a smooth domain.
By the standard direct method ( see \cite[Lemma2.7]{BQZ24}),  we obtain that   \[
\liminf_{r\to\infty}\frac{\text{dist}(x_r,\partial\Omega_r)}{\tau_r}>0.\]
As a result, under the condition (G3), let $r \to \infty$ in \eqref{equ:241204-e1}, we find that
   $\omega$ is a nonnegative solution of problem
    \[\begin{cases}
        -\Delta\omega =B|\omega|^{\beta-2}\omega,&\text{in}\ \Sigma,\\
        \omega=0&\text{on}\ \partial\Sigma.
    \end{cases}\]
 If \[
\liminf_{r\to\infty}\frac{\text{dist}(x_r,\partial\Omega_r)}{\tau_r}=\infty\]
occurs,  $\Sigma=\mathbb{R}^N$. Since $\beta<2^*$, by the remarkable result \cite{CL91}, the only nonnegative solution of is 0,  which is impossible due to the fact that $\omega(0)=\liminf_{r\to\infty}\omega_r(0)=1$. If \[
\liminf_{r\to\infty}\frac{\text{dist}(x_r,\partial\Omega_r)}{\tau_r}=m>0\]
occurs, we next \textbf{claim} $\Sigma$ is a half space. 
In fact, let  dist$(x_r,\partial\Omega_r)=|x_r-z_r|$ with $z_r\in\partial\Omega_r$, then $\tilde{z}_r=z_r/r\in\partial\Omega$ and $\tilde{x}_r=x_r/r\in\Omega$.  Based on the definition of $\Sigma_r$,  the origin is located at $\tilde{x}_r$ and $\tilde{x}_r-\tilde{z}_r=|x_r-z_r|(1,0,...,0)$ by translating  and rotating  the coordinate
system.
Assume up to a subsequence that $\tilde{z}_r\to z$ with $z\in\Omega$.
And by the smoothness of the domain see \cite[section 6.2]{GT83},  functions $f_r,f:\mathbb{R}^{N-1}\to\mathbb{R}$ are smooth  such that
\[f_r(0)=-|\tilde{x}_r-\tilde{z}_r|=-\frac{|x_r-z_r|}{r},\quad \nabla f_r(0)=0,\]
    \begin{equation}\label{>f}
        \Omega\cap B(\tilde{z}_r,\delta)=\{x\in B(\tilde{z}_r,\delta):x_1>f_r(x_2,...,x_N)\},
    \end{equation}
    and
    \begin{equation}\label{=f}
        \partial\Omega\cap B(\tilde{z}_r,\delta)=\{x\in B(\tilde{z}_r,\delta):x_1=f_r(x_2,...,x_N)\}.
    \end{equation}
    Moreover, \eqref{>f} and \eqref{=f} hold by replacing $\tilde{z}_r$ and $f_r$ with z and f , respectively.
And $f_r\to f$ in $C^1(B_{\frac{\delta}{2}})$, where $B_{\frac{\delta}{2}}\subset\mathbb{R}^{2}$. Consequently,
\[
 \partial\Omega_r\cap B(z_r,r\delta)=\{(rf_r(x'),rx'):|x'|<\delta\}=\big\{\bigg(rf_r(\frac{x'}{r}),x'\bigg):|x'|<r\delta\big\},
\]
where $x'\in\mathbb{R}^{N-1}$. As a consequence, for any $x'\in\mathbb{R}^{N-1}$ and large $r$,
\[
y_r=\bigg(rf_r(\frac{\tau_rx'}{r}),\tau_rx'\bigg)\in\partial\Omega_r\cap B(z_r,r\delta).
\]
Moreover,
\[
\frac{y_r-x_r}{\tau_r}=\bigg(\frac{rf_r(\frac{\tau_rx'}{r})}{\tau_r},x'\bigg).
\]
Note that
\begin{eqnarray*}
\frac{rf_r(\frac{\tau_rx'}{r})}{\tau_r}&=&
\frac{rf_r(\frac{\tau_rx'}{r})-rf_r(0)+rf_r(0)}{\tau_r}
=\frac{r\nabla f_r(\frac{\theta_r\tau_rx'}{r})\cdot\frac{\tau_rx'}{r}-|x_r-z_r|}{\tau_r}  \\
&\to&\nabla f(0)\cdot x'-m=-m.
\end{eqnarray*}
As a result, $\frac{y_r-x_r}{\tau_r}\to(-m,x')$ as $r\to\infty$. Therefore $\Sigma=\{x\in\mathbb{R}^N:x_1>-m\}$, the claim hold. Then we can use the Liouville theorem \cite{EL82}, $\omega=0$ in $\Sigma$, this contradicts $\omega(0)=\liminf_{r\to\infty}\omega_r(0)=1$.  The proof is now complete.
\end{proof}

\begin{lemma}\label{lambda-r>0-betale0}
Assume that the assumptions of Theorem \ref{betale0-e>0-Omega} hold. Then the following results hold.
\begin{description}
\item[(i)] For every $c>0$, problem \eqref{main-eq-omega}  has a solution $(\lambda_r,u_r)$ provided $r>r_c$ where $r_c$ is as in Lemma \ref{betale0-mp-g}. Moreover, $u_r>0$ in $\Omega_r$.
  \item[(ii)] Let $(\lambda_{r,c},u_{r,c})$ is the solution of problem \eqref{main-eq-omega}. Then there exists $\tilde{c}>0$ such that
\[
\underset{r\to\infty}{\liminf}\,\lambda_{r,c}>0\quad\text{for}\ 0<c<\tilde{c}.
\]
\end{description}
\end{lemma}
\begin{proof}
(i) For fixed $c>0$ and  $r>\tilde{r}_c$. Based on the preceding Theorem \ref{beta>0-s-omega} and Lemma \ref{nabla-u-bdd}, we demonstrate that there exist solutions $\{(\lambda_{r,c,s},u_{r,c,s})\}$ to problem \eqref{main-eq-s-omega} for $s\in\left[1/2,1\right]$, and $\{u_{r,c,s}\}\subset S_{r,c}$ is bounded. Now, using a technique similar to the proof of  Theorem \ref{beta>0-s-omega}, we obtain $u_{r,c}\in S_{r,c}$ and $\lambda_{r,c}\in\mathbb{R}$ such that going to a subsequence, $\lambda_{r,c,s}\to\lambda_{r,c}$ and $u_{r,c,s}\to u_{r,c}$ in $H_0^{1}(\Omega_r)$ as $s\to1$. This combined with the strong maximum principle leads to that $u_{r,c}>0$ is a solution of problem \eqref{main-eq-omega}.

For (ii), let $(\lambda_{r,c},u_{r,c})$ is the solution of problem \eqref{main-eq-omega}. The regularity theory of elliptic partial differential equations yields $u_{r,c}\in C(\Omega_r)$.  In view of Lemma \ref{ur-ubdd} we have
\[
\underset{r\to\infty}{\limsup}\,\underset{\Omega_r}{\max}\,u_{r,c}<\infty.
\]
Furthermore, let us \textbf{claim} that there exists $c'_1>0$ such that
\begin{equation}\label{g-alpha>0}
h(c):=\underset{r\to\infty}{\liminf}\,\underset{\Omega_r}{\max}\,u_{r,c}>0
\end{equation}
for any $0<c<c'_1$. Assume
by  contradiction that there exists a sequence $c_k\to0$  as $k\to\infty$ such that
$h(c_k)=0$ for any $k$, that is,
\begin{equation*}
\underset{r\to\infty}{\limsup}\,\underset{\Omega_r}{\max}\,u_{r,c_k}=0\quad\text{for any}\ k.
\end{equation*}
For any fixed $k$, it follows from  $u_{r,c_k}\in S_{r,c_k}$ that for any $t>2$
\begin{equation}\label{urk-m}
\int_{\Omega_r}|u_{r,c_k}|^t\mathrm{d}x\le\big|\underset{\Omega_r}{\max}\,
u_{r,c_k}\big|^{t-2}c_k\to0\quad \text{as}\ r\to\infty.
\end{equation}
Hence, there exists $\bar{r}_k>0$ such that
\[
\Big|\int_{\Omega_r}G(u_{r,c_k})\mathrm{d}x\Big|
<\frac{m_{r,1}(c_k)}{2}\quad\text{for any}\ r\ge\bar{r}_k.
\]
In view of $E_r(u_{r,c_k})=m_{r,1}(c_k)$, we further have that for any large $k$
\begin{equation}\label{nabia-urk-s}
\int_{\Omega_r}|\nabla u_{r,c_k}|^2\mathrm{d}x+\int_{\Omega_r}V u_{r,c_k}^2\mathrm{d}x
\ge m_{r,1}(c_k),\ r\ge\bar{r}_k.
\end{equation}
It follows from \eqref{urk-m}-\eqref{nabia-urk-s} that there exists $r_k\ge \bar{r}_k$ with $\bar{r}_k\to\infty$ as $k\to\infty$ such that
\begin{equation}\label{urk-max}
\underset{k\to\infty}{\limsup}\,\underset{\Omega_{r_k}}{\max}\,u_{r_k,c_k}=0,
\end{equation}
\begin{equation}\label{urk-s-2}
\int_{\Omega_{r_k}}|u_{r_k,c_k}|^t\mathrm{d}x\le\big|\underset{\Omega_{r_k}}{\max}\,
u_{r_k,c_k}\big|^{t-2}c_k\to0\quad \text{as}\ k\to\infty\ \text{for any }t>2.
\end{equation}
and
\begin{equation}\label{nabia-urk-s-2}
\int_{\Omega_{r_k}}|\nabla u_{r_k,c_k}|^2\mathrm{d}x+\int_{\Omega_{r_k}}V u_{r_k,c_k}^2\mathrm{d}x\to\infty\quad\text{as}\ k\to\infty.
\end{equation}
By Equation \eqref{main-eq-omega} and \eqref{urk-s-2}-\eqref{nabia-urk-s-2}
\begin{equation}\label{lambda-k}
\lambda_{r_k,c_k}\to-\infty\ \text{as}\ k\to\infty.
\end{equation}
Now this together with \eqref{main-eq-omega} implies that
\begin{eqnarray*}
&&-\Delta u_{r_k,c_k}+\bigg(\|V\|_\infty
+\frac{\lambda_{r_k,c_k}}{2}u_{r_k,c_k}\bigg)u_{r_k,c_k}\\
&\ge&\bigg(-\frac{\lambda_{r_k,c_k}}{2}+g(u_{r_k,c_k})\bigg)u_{r_k,c_k}\ge0
\end{eqnarray*}
for large $k$. Let $\theta_{r_k}$ be the principal eigenvalue of $-\Delta$ with Dirichlet boundary condition in $\Omega_{r_k}$, and $0<v_{r_k}$ be the corresponding normalized eigenfunction. Then $\theta_{r_k}=\theta_1/r_k^2$ and
\[
\bigg(\frac{\theta_1}{r_k^2}+\|V\|_\infty+\frac{\lambda_{r_k,c_k}}{2}\bigg)
\int_{\Omega_{r_k}}u_{r_k,c_k}v_{r_k}\mathrm{d}x\ge0.
\]
Since $\int_{\Omega_{r_k}}u_{r_k,c_k}v_{r_k}\mathrm{d}x>0$, we have
\[
\frac{\theta_1}{r_k^2}+\|V\|_\infty+\frac{\lambda_{r_k,c_k}}{2}\ge0,
\]
which contradicts \eqref{lambda-k} for large $k$. Hence the \textbf{claim} holds.
Consider $H_0^1(\Omega_r)$ as a subspace of $H^1(\mathbb{R}^{N})$ for any $r>0$. It follows from Lemma \ref{nabla-u-bdd} that the set of solutions $\{u_{r,c}:r>r_c\}$ established in Theorem \ref{beta>0-s-omega} is bounded in $H^1(\mathbb{R}^N)$,  so there exist $u_c\in H^1(\mathbb{R}^N)$ and $\lambda_c\in\mathbb{R}$ such that as $r\to\infty$ up to a subsequence:
\begin{gather*}
  \lambda_{r,c}\to\lambda_c,
  \quad u_{r,c}\rightharpoonup u_c\quad\text{in}\ H^1(\mathbb{R}^N),\\
  u_{r,c}\to u_c\quad\text{in}\ L_{loc}^{t}(\mathbb{R}^N)\ \text{for all}\ 2\le t<2^*,
\end{gather*}
and $u_c$ is a solution of the equation
\[
-\Delta u_c+Vu_c+\lambda_c u_c=g(u_c)\quad \text{in}\ \mathbb{R}^N.
\]
Next, it is explained in two cases.

\textbf{Case1:} $u_c\neq0$ for $c>0$ small.
Therefore,
\begin{equation}\label{alpha-eq}
\int_{\mathbb{R}^N}|\nabla u_c|^2\mathrm{d}x+\int_{\mathbb{R}^N}Vu_c^2\mathrm{d}x
+\lambda_c\int_{\mathbb{R}^N}u_c^2\mathrm{d}x
=\int_{\mathbb{R}^N}g(u)u\mathrm{d}x
\end{equation}
and the Pohozaev identity gives
\begin{equation}\label{alpha-eq-pohozaev}
\frac{N-2}{2N}\int_{\mathbb{R}^N}|\nabla u_c|^2\mathrm{d}x
+\frac{1}{2N}\int_{\mathbb{R}^N}\tilde{V}u_c^2\mathrm{d}x
+\frac{1}{2}\int_{\mathbb{R}^N}Vu_c^2\mathrm{d}x
+\frac{\lambda_c}{2}\int_{\mathbb{R}^N}u_c^2\mathrm{d}x
=\int_{\mathbb{R}^N}G(u)\mathrm{d}x.
\end{equation}
Next it follows from \eqref{G-compare},  \eqref{alpha-eq} and \eqref{alpha-eq-pohozaev} that
\begin{eqnarray*}
\bigg(\frac{1}{N}-\frac{\|\tilde{V}_+\|_{\frac{N}{2}}S^{-1}}{2N}\bigg)\int_{\mathbb{R}^N}|\nabla u_c|^2\mathrm{d}x&\le&\frac{1}{N}\int_{\mathbb{R}^N}|\nabla u_c|^2\mathrm{d}x
-\frac{1}{2N}\int_{\mathbb{R}^N}\tilde{V}u_c^2\mathrm{d}x\\
&=&\frac{1}{2}\int_{\mathbb{R}^N}g(u)u\mathrm{d}x-\int_{\mathbb{R}^N}G(u)\mathrm{d}x\notag\\
&\le&\frac{\beta-2}{2}\int_{\mathbb{R}^N}G(u)\mathrm{d}x \\
&\le&\frac{C_{N,d}(\beta-2)}{2}C_2\bigg(\int_{\mathbb{R}^N}|u_c|^2\mathrm{d}x\bigg)^{\frac{d(N-2)-2N}{4}}
\bigg(\int_{\mathbb{R}^N}|\nabla u_c|^2\mathrm{d}x\bigg)^{\frac{N(d-2)}{4}}.
\end{eqnarray*}
As a consequence there holds for $u_c\neq0$:
\begin{align}\label{equ:250103-e5}
    \int_{\mathbb{R}^N}|\nabla u_c|^2\mathrm{d}x\ge
\bigg(\frac{(2-\|\tilde{V}_+\|_{\frac{N}{2}}S^{-1})}{NC_{N,d}C_2(\beta-2)}\bigg)
^{\frac{4}{N(d-2)-4}}c^{\frac{d(N-2)-2N}{N(d-2)-4}}.
\end{align}
On the other hands, from the above  inequalities and (G2), we conclude that
\begin{eqnarray}\label{equ:250103-e4}
\frac{2-\beta}{2\beta}\lambda_c\int_{\mathbb{R}^N}u_c^2\mathrm{d}x
&=&\frac{(N-2)\beta-2N}{2N\beta}\int_{\mathbb{R}^N}|\nabla u_c|^2\mathrm{d}x
+\frac{1}{2N}\int_{\mathbb{R}^N}\tilde{V}u_c^2\mathrm{d}x\\
&&\quad+\frac{\beta-2}{2\beta}\int_{\mathbb{R}^N}Vu_c^2\mathrm{d}x
+\frac{1}{\beta}\int_{\mathbb{R}^N}g(u)u\mathrm{d}x-\int_{\mathbb{R}^N}G(u)\mathrm{d}x\notag\\  \nonumber
&\le&\frac{(N-2)\beta-2N}{2N\beta}\int_{\mathbb{R}^N}|\nabla u_c|^2\mathrm{d}x
+\frac{\|\tilde{V}\|_\infty}{2N}c+\frac{\beta-2}{2\beta}\|V\|_\infty c\\ \nonumber
&\le&\frac{(N-2)\beta-2N}{2N\beta}
\bigg(\frac{(2-\|\tilde{V}_+\|_{\frac{N}{2}}S^{-1})}{NC_{N,d}C_2(\beta-2)}\bigg)
^{\frac{4}{N(d-2)-4}}c^{\frac{d(N-2)-2N}{N(d-2)-4}}
\\ \nonumber
&&\quad+\frac{\|\tilde{V}\|_\infty}{2N}c+\frac{\beta-2}{2\beta}\|V\|_\infty c\\  \nonumber
&\to&-\infty\quad\text{as}\ c\to0.
\end{eqnarray}
Therefore,  $\lambda_c>0$ for $c>0$ small.

\textbf{Case2: }There is a sequence $c_n\to0$
such that $u_{c_n}=0$. In view of \eqref{g-alpha>0}, $u_{c_n}=0$ for
any $c_n\in\left(0,c'_1\right)$. Let $z_{r,c_n}\in\Omega_r$ be such that $u_{r,c_n}(z_{r,c_n})=\max_{\Omega_r}u_{r,c_n}$, it holds $|z_{r,c_n}|\to\infty$ as $r\to\infty$. Otherwise, there exists $z_0\in\mathbb{R}$ such
that $z_{r,c_n}\to z_0$, and hence $u_{c_n}(z_0)\ge d_{c_n}> 0$. This contradicts
$u_{c_n}=0$.  Moreover, dist$(z_{r,c_n},\partial\Omega_r)\to\infty$ as $r\to\infty$ by an argument similar to that in Lemma \ref{lambda-r>0-betale0}.
Now, for $n$ fixed, let $v_r(x)=u_{r,c_n}(x+z_{r,c_n})$ for any $x\in\Sigma_r:=\{x\in\mathbb{R}^N:x+z_{r,c_n}\in\Omega_r\}$.
It follows from Lemma \ref{nabla-u-bdd} that $v_r$ is bounded in $H^1(\mathbb{R}^N)$, and there is $v\in H^1(\mathbb{R}^N)$ with $v\neq0$ such that $v_r\rightharpoonup v$ as $r\to\infty$.
Observe that for every $\varphi\in C_c^\infty(\Omega_r)$, there is  $r$ large such that $\varphi(\cdot-z_{r,c})\in C_c^\infty(\Omega_r)$ due to dist$(z_{r,c_n},\partial\Omega_r)\to\infty$ as $r\to\infty$. It then follows that
\begin{eqnarray}\label{varphi-z}
&&\int_{\Omega_r}\nabla u_{r,c_n}\nabla\varphi(\cdot-z_{r,c_n})\mathrm{d}x
+\int_{\Omega_r}Vu_{r,c_n}\varphi(\cdot-z_{r,c_n})\mathrm{d}x
+\lambda_{r,c_n}\int_{\Omega_r}u_{r,c_n}\varphi(\cdot-z_{r,c_n})\mathrm{d}x\notag\\
&=&\int_{\Omega_r}g(u_{r,c_n})u_{r,c_n}\varphi(\cdot-z_{r,c_n})\mathrm{d}x.
\end{eqnarray}
Since $|z_{r,c_n}|\to\infty$ as $r\to\infty$, we have
\begin{eqnarray}\label{v-0}
\Big|\int_{\Omega_r}Vu_{r,c_n}\varphi(\cdot-z_{r,c_n})\mathrm{d}x\Big|
&\le&\int_{\text{Supp}\varphi}\Big|V(\cdot+z_{r,c_n})v_{r}\varphi\Big|\mathrm{d}x\notag\\
&\le&\|v_{r}\|_{2^*}\|\varphi\|_{2^*}
\bigg(\int_{\text{Supp}\varphi}|V(\cdot+z_{r,c_n})|^{\frac{3}{2}}\mathrm{d}x\bigg)^{\frac{2}{3}}\notag\\
&\le&\|v_{r}\|_{2^*}\|\varphi\|_{2^*}
\bigg(\int_{\mathbb{R}^N\setminus B_{\frac{|z_{r,c_n}|}{2}}}|V(\cdot+z_{r,c_n})|
^{\frac{3}{2}}\mathrm{d}x\bigg)^{\frac{2}{3}}\notag\\
&\to&0\quad\text{as}\ r\to\infty.
\end{eqnarray}
Letting $r\to\infty$ in \eqref{varphi-z} we obtain for $\varphi\in C_c^\infty(\mathbb{R}^{N})$:
\[
\int_{\mathbb{R}^N}\nabla v\cdot\nabla\varphi\mathrm{d}x
+\lambda_{c_n}\int_{\mathbb{R}^N}v\varphi\mathrm{d}x
=\int_{\mathbb{R}^N}g(v)\varphi\mathrm{d}x.
\]
 Now we argue as in the case that $u_c\neq0$ above, there exists $\tilde{c}<c'_1$ such that $\lambda_c>0$ for any $c\in\left(0,\tilde{c}\right)$, the proof is now complete.
\end{proof}

\noindent\textbf{Proof of Theorem \ref{betale0-e>0-Omega}:} The proof is an immediate consequence of Theorem \ref{beta>0-s-omega} and Lemma \ref{lambda-r>0-betale0}.\qed

\section{Normalized solution in $\mathbb{R}^N$}

The aim of this section is to obtain the existence of normalized solution to Eq \eqref{main-eq}-\eqref{constraint} in  $\mathbb{R}^N$. For this purpose, under the assumptions of  Theorem \ref{betale0-e>0-Omega}, we first obtain the solutions  $(\lambda_r,u_r)$ to \eqref{main-eq-s-omega} with $\Omega_r=B_r$, see Theorem \ref{betale0-e>0-Omega}. Then by Lemma \ref{lambda-r>0-betale0}, we conclude that $u_r$ is bounded uniformly in $r$ and $\liminf_{r\to\infty}\lambda_r>0$.
The next lemma is to analyze the compactness of normalized solutions $u_r$ for \eqref{main-eq-s-omega} with $\Omega=B_r$ as  $r\to\infty$.

\begin{lemma}\label{rinftycom}
Under the assumptions of   Theorem \ref{betale0-e>0-Omega}, let $\{(\lambda_r,u_r)\}$ be a sequence of solutions of  \eqref{main-eq-omega} with $\Omega_r=B_r$. Then there exist a subsequence (
still denoted by $\{(\lambda_r,u_r)\}$), a $u_0\in H^1(\mathbb{R}^N)$ satisfying
\[
  -\Delta u+Vu+\lambda u=g(u)\quad\text{in}\ \mathbb{R}^N,
\]
$l\in\mathbb{N}\cup\{0\}$, nontrivial solutions  $w^1,...,w^l\in H^1(\mathbb{R}^N)\setminus\{0\}$ of the following problem
  \begin{equation}\label{omega}
  -\Delta u+\lambda u=g(u) \quad\text{in}\ \mathbb{R}^N
  \end{equation}
such that
\[
\lambda_r\to\lambda>0,\quad\Big\|u_r-u_0-\Sigma_{k=1}^{l}w^k(\cdot-z_r^k)\Big\|_{H^1}\to0\quad \text{as}\ r\to\infty,
\]
where $\{z^k_r\}\subset\mathbb{R}^{N}$ with $z^k_r\in B_r$  satisfying
\begin{equation}\label{zk-dist}
|z^k_r|\to\infty,\quad dist(z^k_r,\partial B_r)\to\infty,\quad|z^k_r-z^{k'}_r|\to\infty
\end{equation}
for any $k,k'=1,...,l$ and $k\neq k'$.
\end{lemma}

\begin{proof}
   Recall the fact that $u_r$ is bounded uniformly in $r$, there exist $ u_0\in H^1(\mathbb{R}^N)$ and $\lambda>0$ such that, up to a subsequence,
\[
\lambda_r\to\lambda,\quad u_r\rightharpoonup u_0\neq 0\quad\text{in}\ H^1(\mathbb{R}^N).
\]
Moreover, $u_0$ is a solution of
\[
-\Delta u+Vu+\lambda u=g(u) \quad\text{in}\ \mathbb{R}^N.
\]
Set $\nu_r^1=u_r-u_0$, then $\nu_r^1\rightharpoonup0$ as $r\to\infty$ in $H^1(\mathbb{R}^N)$.
We divide the proof into three steps:

\textbf{Step1: } The vanishing case occurs. That is,
\[
\underset{r\to\infty}{\limsup}\underset{z\in\mathbb{R}^N}{\sup}\int_{B(z,1)}|\nu_r^1|^2\mathrm{d}x=0.
\]
By the Brezis-Lieb lemma, the concentration compactness principle and the Vitali convergence theorem,
\begin{eqnarray*}
o_r(1)&=&\int_{\mathbb{R}^N}g(\nu_r)\nu_r\mathrm{d}x\\
&=&\int_{\mathbb{R}^N}g(u_r)u_r\mathrm{d}x
-\int_{\mathbb{R}^N} g(u_0)u_0\mathrm{d}x+o_r(1)\\
&=&\int_{\mathbb{R}^N}|\nabla u_r|^2\mathrm{d}x
+\int_{\mathbb{R}^N}Vu_r^2\mathrm{d}x
+\lambda\int_{\mathbb{R}^N}|u_r|^2\mathrm{d}x\\
&&\quad-\int_{\mathbb{R}^N}|\nabla u_0|^2\mathrm{d}x-\int_{\mathbb{R}^N}Vu_0^2\mathrm{d}x
-\lambda\int_{\mathbb{R}^N}|u_0|^2\mathrm{d}x\\
&=&\int_{\mathbb{R}^N}|\nabla (u_r-u_0)|^2\mathrm{d}x
+o_r(1).
\end{eqnarray*}
As a consequence $\nu_r^1=u_r-u_0\to 0$ in $H^1(\mathbb{R}^N)$, and the lemma holds for $l=0$.

\textbf{ Step2: }The vanishing does not occur. That is, there exists $z_r^1\in B_r$ such that
\begin{equation}\label{vrd}
\int_{B(z_r^1,1)}|\nu_r^1|^2\mathrm{d}x>m>0.
\end{equation}
We first claim that $|z_r^1|\to\infty$ as $r\to\infty$. Assume  by contradiction that  there is $R>0$ such that $|z_r^1|<R$,
\[
\int_{B_{R+1}}|\nu_r^1|^2\mathrm{d}x
\ge\int_{B(z_r^1,1)}|\nu_r^1|^2\mathrm{d}x>m,
\]
contradicting $\nu_r^1\rightharpoonup0$ as $r\to\infty$. Then $|z_r^1|\to\infty$ as $r\to\infty$.
Next, we \textbf{claim} that
\begin{equation}\label{dist}
\text{dist}(z_{r}^1,\partial B_r)\to\infty,\quad \text{as}\ r\to\infty.
\end{equation}
 Indeed, assume by contradiction that
$\liminf_{r\to\infty}\text{dist}(z_{r}^1,\partial B_r)=m_0<\infty$.
 It follows from \eqref{g-alpha>0} that $m_0>0$.
Let $v_r(x):=\nu_{r}^1(x+z_{r}^1)$ for any $x\in\Sigma_r:=\{x\in\mathbb{R}^N:x+z_{r}^1\in\Omega_r\}$. Then $v_r$ is bounded in $H^1(\mathbb{R}^N)$, and there is $v\in H^1(\mathbb{R}^N)$ such that $v_r\rightharpoonup v$ as $r\to\infty$. By the regularity theory of elliptic partial differential equations and $\liminf_{r\to\infty}\nu_{r}^1(z_{r}^1)\ge m_0> 0$, we deduce that $v(0)\ge m_0>0$. Assume without loss of generality that, up to a subsequence,
\[
\underset{r\to\infty}{\lim}\frac{z_{r}^1}{|z_{r}^1|}=e_1.
\]
Setting
\[
\Sigma=\Big\{x\in\mathbb{R}^N:x\cdot e_1<l\Big\}=\Big\{x\in\mathbb{R}^N:x_1<l\Big\},
\]
we have $\varphi(\cdot-z_{r}^1)\in C_c^\infty(\Omega_r)$ for any $\varphi\in C_c^\infty(\Omega_r)$ and $r$ large enough. It then follows that
\begin{eqnarray}\label{eq-varphi}
&&\int_{\Omega_r}\nabla \nu_{r}^1\nabla\varphi(\cdot-z_{r}^1)\mathrm{d}x
+\int_{\Omega_r}V\nu_{r}^1\varphi(\cdot-z_{r}^1)\mathrm{d}x
+\lambda_{r,\alpha_n}\int_{\Omega_r}\nu_{r}^1\varphi(\cdot-z_{r}^1)\mathrm{d}x\notag\\
&=&\int_{\Omega_r}g(\nu_{r}^1)\varphi(\cdot-z_{r}^1)\mathrm{d}x.
\end{eqnarray}
Since $|z_r^1|\to\infty$ as $r\to\infty$, we have
\begin{eqnarray}\label{V-vanish}
\Big|\int_{\Omega_r}V\nu_{r}^1\varphi(\cdot-z_{r}^1)\mathrm{d}x\Big|
&\le&\int_{\text{Supp}\varphi}\Big|V(\cdot+z_r^1)v_{r}\varphi\Big|\mathrm{d}x\notag\\
&\le&\|v_{r}\|_{2^*}\|\varphi\|_{2^*}
\bigg(\int_{\text{Supp}\varphi}|V(\cdot+z_r^1)|^{\frac{3}{2}}\mathrm{d}x\bigg)^{\frac{2}{3}}\\
&\le&\|v_{r}\|_{2^*}\|\varphi\|_{2^*}
\bigg(\int_{\mathbb{R}^N\setminus B_{\frac{|z_r^1|}{2}}}|V(\cdot+z_r^1)|
^{\frac{3}{2}}\mathrm{d}x\bigg)^{\frac{2}{3}}\to0\quad\text{as}\ r\to\infty\notag.
\end{eqnarray}
Letting $r\to\infty$ in \eqref{eq-varphi}, we obtain for $\varphi\in C_c^\infty(\Sigma)$:
\[
\int_{\Sigma}\nabla v\cdot\nabla\varphi\mathrm{d}x
+\lambda \int_{\Sigma}v\varphi\mathrm{d}x
=\int_{\Sigma}g(v)\varphi\mathrm{d}x.
\]
Thus $v\in H^1(\mathbb{R}^N)$ is a weak solution of the equation
\begin{equation}\label{eq-w-alpha-ge}
-\Delta v+\lambda v=g(v)\quad\text{in}\ \Sigma.
\end{equation}
Hence we obtain a nontrivial nonnegative solution of \eqref{eq-w-alpha-ge} on a half space, which
is impossible (see e.g., \cite{EL82}). This proves that dist$(z_r^1,\partial\Omega_r)\to\infty$ as $r\to\infty$. A similar argument as above shows that \eqref{eq-w-alpha-ge} holds for $\Sigma=\mathbb{R}^N$.
Thus
\begin{equation}\label{alpha-eq-j-ge}
\int_{\mathbb{R}^N}|\nabla v|^2\mathrm{d}x
+\lambda\int_{\mathbb{R}^N}|v|^2\mathrm{d}x
=\int_{\mathbb{R}^N}g(v)\mathrm{d}x.
\end{equation}
Thus $\omega_r^1:=\nu_r^1(\cdot+z_r^1)\rightharpoonup\omega^1$ in $H^1(\mathbb{R}^N)$, and after a calculation similar to \eqref{v-0}, we get that  $\omega^1$ satisfies
\[
-\Delta \omega^1+\lambda\omega^1=g(\omega^1)\quad\text{in}\ \mathbb{R}^{N}.
\]
Moreover, by \eqref{vrd},
\[
\int_{B_1}|\omega^1|^2\mathrm{d}x
=\underset{r\to\infty}{\lim}\int_{B_1}|\nu_r^1(x+z_r^1)|^2\mathrm{d}x
=\underset{r\to\infty}{\lim}\int_{B(z_r^1,1)}|\nu_r^1(x)|^2\mathrm{d}x>m>0.
\]
This implies $\omega^1\neq0$. Setting $\nu_r^2(x)=\nu_r^1(x)-\omega^1(x-z_r^1)$, then $\nu_r^2\rightharpoonup0$ as $r\to\infty$ in $H^1(\mathbb{R}^N)$.
\textbf{Case1:} If
\[
\underset{r\to\infty}{\limsup}\underset{z\in\mathbb{R}^N}{\sup}
\int_{B(z_r^1,1)}|\nu_r^2(x)|^2\mathrm{d}x=0
\]
occurs, then we stop and  go to \textbf{step 1}. Consequently the lemma holds for $l=1$.
\textbf{Case2:} If there exists $z_r^2\in B_r$ such that
\[
\int_{B(z_r^2,1)}|\nu_r^2(x)|^2\mathrm{d}x>m>0,
\]
we first \textbf{claim} that $|z_r^1-z_r^2|\to\infty$ as $r\to\infty$. We suppose by the contrary that there is $R>0$ such that $|z_r^1-z_r^2|<R$ for any large $r$. Then
\begin{eqnarray}
\int_{B(z_r^2,1)}|\nu_r^2(x)|^2\mathrm{d}x&\le&\int_{B(z_r^1,R+1)}|\nu_r^2(x)|^2\mathrm{d}x\\
&=&\int_{B(z_r^1,R+1)}|\nu_r^1(x)-\omega^1(x-z_r^1)|^2\mathrm{d}x\notag\\
&=&\int_{B(0,R+1)}|\omega_r^1(x)-\omega^1(x)|^2\mathrm{d}x\notag\\
&\to&0\quad\text{as}\ r\to\infty,
\end{eqnarray}
which is impossible. Then the claim is valid.  Base on this, similar to \eqref{dist} we further have that  dist$(z_r^2,\partial\Omega_r)\to\infty$ as $r\to\infty$, \eqref{zk-dist} holds, and there is $\omega^2\neq0$ satisfying \eqref{omega}. Let $\nu_r^3=\nu_r^2(x)-\omega^2(x-z_r^2)$, repeating the above process.

\textbf{Step3:} we try to prove that the process  in \textbf{Step2} can be repeated at most finitely many times. For any solution $\omega\neq0$ of \eqref{omega},
let us  first \textbf{claim} $\|\nabla\omega\|_2^2$ is  bounded from below. Indeed, from \eqref{omega}, we have
\[
\int_{\mathbb{R}^N}|\nabla \omega|^2\mathrm{d}x
+\lambda\int_{\mathbb{R}^N}|\omega|^2\mathrm{d}x
=\int_{\mathbb{R}^N}g(\omega)\omega\mathrm{d}x.
\]
This together with \eqref{G-compare}, \eqref{g-G} and  the Pohozaev identity
leads to
\begin{eqnarray*}
\frac{1}{N}\int_{\mathbb{R}^N}|\nabla \omega|^2\mathrm{d}x
&=&\frac{1}{2}\int_{\mathbb{R}^N}g(\omega)\omega\mathrm{d}x
-\int_{\mathbb{R}^N}G(\omega)\omega\mathrm{d}x\\
&\le&(\frac{\beta}{2}-1)C_2\int_{\mathbb{R}^N}|\omega|^{d}\mathrm{d}x\\
&\le&(\frac{\beta}{2}-1)C_{N,d}C_2c^{\frac{2d-N(d-2)}{2}}\bigg(\int_{\mathbb{R}^N}|\nabla \omega|^2\mathrm{d}x\bigg)^{\frac{N(d-2)}{4}}.
\end{eqnarray*}
Then $\|\nabla\omega\|_2^2$ is  bounded from below due to $2+\frac{4}{N}<d<2^*$.
We suppose that it can be repeated at most l times. Then
\begin{equation}
\underset{r\to\infty}{\lim}\|u_r\|_{H^1}\ge\|u_0\|_{H^1}
+\Sigma_{k=1}^{l-1}\|\omega^k(\cdot-z_r^k)\|_{H^1}
\end{equation}
If $l=\infty$, this contradicts with the fact that $u_r$ is bounded uniformly in $r$.
\end{proof}

\noindent\textbf{Proofs of Theorems \ref{beta>0-e<0-rn-}:}   This part follows a standard technique, but let us quickly explain for the reader¡¯s convenience.
Let $\{(\lambda_r,u_r)\}$ be a sequence of solutions of \eqref{main-eq-s-omega} with $\Omega_r=B_r$.
In view of Lemma \ref{rinftycom}, if $l=0$,  $u_r\to u_0$ as $r\to\infty$ in $H^1(\mathbb{R}^N)$, then Theorems \ref{beta>0-e<0-rn-}  hold. Next, we just need to state that the case $l>0$ is impossible  through a discussion similar to \cite{BQZ24}.

In fact, if $l>0$, without loss of generality, we assume that
$|z^1_r|\le\min\{|z_r^k|:k=1,...,l\}$.
Let
\[
\frac{|z_r^k-z_r^1|}{|z_r^1|}\to d_k\in\left[0,\infty\right],
\quad K=\{k:d_k\neq0\ \text{for}\ k=1,...,l\},\]
and set $v^*=\frac{1}{4}\min\{1,\rho,d^*\}$ with $d^*={\min}_{k\in K}d_k>0$. Fixed $\delta\in\left(0,v^*\right)$ and consider the annulus
\[
A_r=B\bigg(z_r^1,\frac{3}{2}\delta|z_r^1|\bigg)\setminus B\bigg(z_r^1,\frac{1}{2}\delta|z_r^1|\bigg).
\]
It is easy to verify
 that $\text{dist}(z_r^k,A_r)\ge\frac{1}{4}\delta|z_r^1|$ for any $k=1,...,l$ and large $r$.
 By Lemma \ref{rinftycom}, one has that  for any $2\le s< 2^*$,
\[
\|u_r\|_{L^s(A_r)}=\Big\|u_0+\Sigma_{k=1}^{l}\omega^k(\cdot-z_r^k)\Big\|_{L^s(A_r)}
\le\|u_0\|_{L^s(A_r)}+\Sigma_{k=1}^{l}\|\omega^k(\cdot-z_r^k)\|_{L^s(A_r)}.
\]
Then
\[\|u_0\|_{L^s(A_r)}=\int_{B(z_r^1,\frac{3}{2}\delta|z_r^1|)}|u_0|^s\mathrm{d}x
\le\int_{ B^c_{\delta|z_r^1|/2}}|u_0|^s\mathrm{d}x\to0\]
and
\[
\|\omega^k(\cdot-z_r^k)\|_{L^s(A_r)}
=\int_{A_r}|\omega^k(\cdot-z_r^k)|^s\mathrm{d}x
\le\int_{B^c_{\delta|z_r^1|/4}}|\omega^k|^s\mathrm{d}x\to0.
\]
As a result,  $\|u_r\|_{L^s(A_r)}\to0$ as $r\to\infty$. Observe that  $u_r$ satisfies
\begin{equation}\label{eq-la}
    -\Delta u_r+Vu_r+\lambda u_r\le g(u_r)\quad\text{in}\ \mathbb{R}^N
\end{equation}
for large r, where $\lambda:=\frac{1}{2}\liminf_{r\to\infty}\lambda_r>0$ due to Lemma \ref{rinftycom}. Setting
\[
R_m=\overline{B(z_r^1,\frac{3}{2}\delta|z_r^1|-m)}\setminus B(z_r^1,\frac{1}{2}\delta|z_r^1|+m)\quad\text{for}\ n\in\mathbb{N}^+,
\]
by \cite[Theorem 8.17]{GT83}, one has that  there is constant $C>0$ independent of r such that for large r,
\begin{equation}\label{lambda}
C_2\|u_r\|_{L^\infty(R_1)}^{d-2}\le C\|u_r\|_{L^2(A_r)}^{d-2}<\frac{\lambda}{4}.
\end{equation}
Let us set $\phi_m=\xi_m(|x-z_r^1|)$, where $\xi_m\in C_c^\infty(\mathbb{R},\left[0,1\right])$ is a cut-off function with  $|\xi_m'(t)\le4|$ for any $t\in R$, and
\[
\xi_m(t)=
\begin{cases}
    1&\text{if}\ \frac{1}{2}\delta|z_r^1|+m<t<\frac{3}{2}\delta|z_r^1|-m,\\
    0&\text{if}\ t<\frac{1}{2}\delta|z_r^1|+m-1\ \text{or}\  t>\frac{3}{2}\delta|z_r^1|-m+1.
\end{cases}
\]
Testing \eqref{eq-la} with $\phi_m^2u_r$, we have
\begin{eqnarray}\label{text-right}
   &&8\int_{R_{m-1}\setminus R_{m}}|u_r|\cdot|\nabla u_r|\mathrm{d}x\notag\\
   &\ge& \int_{R_{m-1}}|\nabla u_r|^2\phi_m^2\mathrm{d}x
    +\int_{R_{m-1}}V| u_r|^2\phi_m^2\mathrm{d}x
    +\lambda\int_{R_{m-1}}| u_r|^2\phi_m^2\mathrm{d}x\notag\\
    &&\quad-\int_{R_{m-1}}g(u_r)u_r\phi_m^2\mathrm{d}x.
\end{eqnarray}
Since $|z_r^1|\to\infty$ and $\lim_{|x|\to\infty}V(x)=0$ there exists $\bar{r}$ such that $V(x)\ge-\frac{\lambda}{4}$ for any  $x\in A_r$ with $r>\bar{r}$. This together with \eqref{lambda}, we get that
\begin{eqnarray}\label{text-left}
    &&\min\big\{1,\lambda/2\big\}\bigg(\int_{R_{m-1}}|\nabla u_r|^2\phi_m^2\mathrm{d}x+\int_{R_{m-1}}| u_r|^2\phi_m^2\mathrm{d}x\bigg)\notag\\
    &\le&\int_{R_{m-1}}|\nabla u_r|^2\phi_m^2\mathrm{d}x+\lambda/2\int_{R_{m-1}}| u_r|^2\phi_m^2\mathrm{d}x\notag\\
    &\le& \int_{R_{m-1}}|\nabla u_r|^2\phi_m^2\mathrm{d}x
    +\int_{R_{m-1}}V| u_r|^2\phi_m^2\mathrm{d}x
    +\lambda\int_{R_{m-1}}| u_r|^2\phi_m^2\mathrm{d}x\notag\\
    &&\quad-\int_{R_{m-1}}g(u_r)u_r\phi_m^2\mathrm{d}x.
\end{eqnarray}
It follows from \eqref{text-right}-\eqref{text-left} that
\[\int_{R_{m}}|\nabla u_r|^2\phi_m^2\mathrm{d}x+\int_{R_{m}}| u_r|^2\phi_m^2\mathrm{d}x\le\kappa\bigg(\int_{R_{m}\setminus R_{m-1}}|\nabla u_r|^2\phi_m^2\mathrm{d}x+\int_{R_{m}\setminus R_{m-1}}|u_r|^2\phi_m^2\mathrm{d}x\bigg),\]
where $\kappa=4\max\{1,\frac{2}{\lambda}\}$.
From this, one has that
\[
a_m\le\theta a_{m-1}\le \theta^m\underset{r}{\max}\|u_r\|_{H^1}=\big(\underset{r}{\max}\|u_r\|_{H^1}\big)e^{mln\theta}=:\tilde{A}e^{mln\theta},
\]
where \[
a_m:=\int_{R_{m}}|\nabla u_r|^2\mathrm{d}x+\int_{R_{m}}| u_r|^2\mathrm{d}x\quad\text{and}\quad \theta=\frac{\kappa}{\kappa+1}.
\]
Let
\[
D_r:=B(z_r^1,\delta|z_r^1|+1)\setminus B(z_r^1,\delta|z_r^1|-1)\subset R_M
\]
with
$M=\bigg[\frac{\delta|z_r^1|}{2}\bigg]-1.$
We further have
\[\int_{D_r}|\nabla u_r|^2\mathrm{d}x+\int_{D_r}| u_r|^2\mathrm{d}x
\le \tilde{A}e^{Mln\theta}\le\tilde{A}e^{\frac{\delta|z_r^1|}{2}ln\theta}.\]
It follows from  \cite[Theorem 8.17 ]{GT83} that
\begin{equation}\label{e-u}
|u_r(x)|\le C\|u_r\|_{L^2(D_r)}\le A e^{-c|z_r^1|}
\end{equation}
for any $x\in D_{r,\frac{2}{3}}:=\delta|z_r^1|-\frac{2}{3}<|x-z_r^1|<\delta|z_r^1|+\frac{2}{3}$ and large r; here A, c, C
are independent of r.
From the $L^p$-estimates of elliptic partial differential equations \cite[Theorem 9.11]{GT83},  we further conclude that
\begin{equation}\label{e-nabla-u}
    |\nabla u_r(x)|\le A e^{-c|z_r^1|}
\end{equation}
for any $x$ with $\delta|z_r^1|-\frac{1}{2}<|x-z_r^1|<\delta|z_r^1|+\frac{1}{2}$.

With \eqref{e-u}-\eqref{e-nabla-u} in hand,  let $\Gamma_1=\partial B(z_r^1,\delta|z_r^1|)\cap B_r$ and $\Gamma_2=B(z_r^1,\delta|z_r^1|)\cap\partial B_r$. Multiplying both sides
of \eqref{main-eq-omega} with $z_r^1\cdot\nabla u_r$ and integrating, we obtain
\begin{eqnarray}\label{a-a2}
    \frac{1}{2}\int_{B(z_r^1,\delta|z_r^1|)\cap B_r}(z_r^1\cdot\nabla V)u_r^2\mathrm{d}\sigma&=&\frac{1}{2}\int_{\Gamma_1\cup\Gamma_2}(z_r^1\cdot\textbf{n})|\nabla u_r|^2\mathrm{d}\sigma\notag\\
    &&\quad-\int_{\Gamma_1\cup\Gamma_2}(\textbf{n}\cdot\nabla u_r)(z_r^1\cdot\nabla u_r)\mathrm{d}\sigma\notag\\
    &&\quad-\int_{\Gamma_1\cup\Gamma_2}(z_r^1\cdot\textbf{n})
    (\frac{Vu_r^2}{2}-G(u_r)+\frac{\lambda|u_r|^2}{2})\mathrm{d}\sigma\notag\\
    &=:&\mathcal{A}_1+\mathcal{A}_2,
\end{eqnarray}
where $\textbf{n}$ is the outward unit normal vector on $\partial{B(z_r^1,\delta|z_r^1|)\cap B_r}$,
\begin{eqnarray*}
    \mathcal{A}_i:&=&\frac{1}{2}\int_{\Gamma_i}(z_r^1\cdot\textbf{n})|\nabla u_r|^2\mathrm{d}\sigma\notag\\
    &&\quad-\int_{\Gamma_i}(\textbf{n}\cdot\nabla u_r)(z_r^1\cdot\nabla u_r)\mathrm{d}\sigma\notag\\
    &&\quad-\frac{1}{2}\int_{\Gamma_i}(z_r^1\cdot\textbf{n})
    (\frac{Vu_r^2}{2}-G(u_r)+\frac{\lambda|u_r|^2}{2})\mathrm{d}\sigma
\end{eqnarray*}
for $i=1,2$.
As we know $z_r^1\cdot\textbf{n}(x)>0$ for any $x\in\Gamma_2\neq\emptyset$ by the choice of $\delta$.
From the fact that
\[
u_r(x)>0,\quad \text{for any }x\in B_r(x),\  \text{and } u_r(x)=0\ \text{for any }x\in\partial B_r(x),
\]
we deduce $\nabla u_r(x)=-|\nabla u_r(x)|\textbf{n}(x)=-|\nabla u_r(x)|\frac{x}{|x|}$ for any  $x\in\Gamma_2\subset\partial B_r$, which leads to
\begin{equation}\label{a2-le0}
\mathcal{A}_2=-\frac{1}{2}\int_{\Gamma_2}(z_r^1\cdot\textbf{n})|\nabla u_r|^2\mathrm{d}\sigma\le0.
\end{equation}
In view of \eqref{e-u}-\eqref{e-nabla-u},  there is $\tau_0>0$ independent of r such that
\[
\underset{r\to\infty}{\limsup} e^{\tau_0|z_r^1|}\mathcal{A}_1=0.
\]
But it follows from Lemma \ref{rinftycom} that
\[
\underset{r\to\infty}{\limsup}\int_{B(z_r^1,\delta|z_r^1|)\cap B_r}u_r^2\mathrm{d}x\ge\|\omega^1\|_2^2>0.
\]
Moreover, the definition of $\delta$ and the assumption $(V_1)$ imply
\begin{eqnarray*}
    &&\underset{r\to\infty}{\lim}
e^{\tau_0|z_r^1|}\int_{B(z_r^1,\delta|z_r^1|)\cap B_r}(z_r^1\cdot\nabla V)u_r^2\mathrm{d}x\\
& \ge& \underset{r\to\infty}{\lim}
    e^{\tau_0|z_r^1|}\bigg(\underset{x\in B(z_r^1,\delta|z_r^1|)}{\inf}(z_r^1\cdot\nabla V)\bigg)\int_{B(z_r^1,\delta|z_r^1|)\cap B_r}u_r^2\mathrm{d}x\\
    &>&0.
\end{eqnarray*}
This contradicts \eqref{a-a2}-\eqref{a2-le0}. The proof is now complete.

		\hfill$\Box$

		\vspace{2cm}
	
\noindent\textbf{Acknowledgement}\ \ 

\bibliographystyle{amsplain}

\end{document}